\documentclass{article}
\usepackage[a4paper]{geometry}
\usepackage{amsmath}
\usepackage{amssymb}
\usepackage{amsthm}
\usepackage{mathrsfs}
\usepackage{mathtools}

\usepackage{tikz}
\usepackage{cleveref}
\usepackage{tikz-cd}
\usepackage{mathtools}

\usepackage{placeins}

\usepackage{float}

\usepackage{comment}

\usepackage{enumerate}
\usepackage[numbers]{natbib}
\theoremstyle{definition}
\newtheorem{thm}{Theorem}[section]
\crefname{thm}{Theorem}{Theorems}
\newtheorem{cor}[thm]{Corollary}
\newtheorem{prop}[thm]{Proposition}
\crefname{prop}{Proposition}{Propositions}
\newtheorem{lem}[thm]{Lemma}
\crefname{lem}{Lemma}{Lemmas}

\newtheorem{conj}[thm]{Conjecture}

\crefname{defn}{Definition}{Definitions}

\newtheorem{rmk}[thm]{Remark}

\newtheorem*{ack*}{Acknowledgements}

\usepackage{titling}

\title{From Brunn-Minkowski to Pr\'ekopa-Leindler and Borell-Brascamp-Lieb: discrete inequalities.}
\author{Peter van Hintum}
\date{}							
\allowdisplaybreaks

\begin{document}
\maketitle
\begin{center}
    \textit{To Professor K\'aroly Bezdek, for his 70th birthday!}
    \vspace*{1cm}
\end{center}

\begin{abstract}
We consider a general way to obtain Pr\'ekopa-Leindler and Borell-Brascamp-Lieb type inequalities from Brunn-Minkowski type inequalities and provide numerous examples. We use the same heuristic to prove a discrete version of the Pr\'ekopa-Leindler and Borell-Brascamp-Lieb inequalities for functions over $\mathbb{Z}^d$. These are the functional extensions of the discrete Brunn-Minkowski inequality conjectured by Ruzsa and recently established by Keevash, Tiba, and the author.
\end{abstract}

\section{Introduction}
Establishing discrete counterparts to fundamental results from convex geometry and analysis is an active area of research, e.g. John's theorem \cite{fritz1948extremum,tao2006additive,tao2008john, berg2019discrete, van2024sharp}, Klartag and Lehec's Slicing Theorem (formerly Bourgain's Slicing Conjecture)\cite{klartag2025affirmative, alexander2017discrete, regev2016note,freyer2022bounds,freyer2024polynomial}, and the Brunn-Minkowski inequality \cite{bollobas1996sums,gardner2001brunn,ollivier2012curved,cifre2018discrete,boroczky2020triangulations,iglesias2020brunn,matolcsi2022analytic,green2022weighted,van2025ruzsa} and its stability \cite{figalli2021sharp,van2021sharp,van2023sets}. The reverse direction seems to have received less attention \cite{van2026locality}. In this note, we prove discrete versions of the Pr\'ekopa-Leindler and Borell-Brascamp-Lieb inequalities (\Cref{mainthm}) which can be seen as the functional versions of the recently established \cite{van2025ruzsa} discrete Brunn-Minkowski inequality as conjectured by Ruzsa \cite{ruzsa2006additive}.

To understand the context of the proof, we examine a rather general method of obtaining Borell-Brascamp-Lieb type inequalities from Brunn-Minkowski type inequalities. A range of applications of this method will be given. This method does not directly extend to the discrete context in which we wish to apply it, but will provide the underlying heuristic of the proof.

\subsection{Inequalities in Euclidean space}
The fundamental Brunn-Minkowski inequality from convex geometry relates the volume of sets $A,B\subset\mathbb{R}^d$ to the volume of its sumset $A+B:=\{a+b: a\in A, b\in B\}$ as
$$|A+B|^{1/d}\geq |A|^{1/d}+|B|^{1/d}.$$
Equivalently, one can normalise and consider $X,Y\subset\mathbb{R}^d$ of equal volume and a parameter $\lambda\in [0,1]$ to find
$$|\lambda X+(1-\lambda)Y|\geq |X|,$$
where $\lambda X:=\{\lambda x: x\in X\}$. This inequality has the following extension to functions called the Pr\'ekopa-Leindler inequality \cite{prekopa1971logarithmic}. Consider integrable functions $f,g,h\colon\mathbb{R}^d\to\mathbb{R}_{\geq 0}$ so that $\int_{\mathbb{R}^d}f(x)dx=\int_{\mathbb{R}^d}g(x)dx$ and for all $x,y\in\mathbb{R}^d$, we have $h(\lambda x+(1-\lambda) y)\geq f(x)^\lambda g(y)^{1-\lambda}$, then
$$\int_{\mathbb{R}^d}h(x)dx\geq \int_{\mathbb{R}^d}f(x)dx.$$

Borell \cite{borell1975convex} and independently Brascamp and Lieb \cite{brascamp1976extensions} showed that one can in fact replace the weighted geometric mean in the lower bound on $h$ by a slighlty smaller harmonic mean. To this end, for $p\in \mathbb{R},a,b\in\mathbb{R}_{\geq0}$, and $\lambda\in(0,1)$, let
$$M_{p,\lambda}(a,b):=\begin{cases} \left(\lambda a^{p}+(1-\lambda)b^{p}\right)^{1/p}&\text{ if }p,a,b\neq0\\
0 &\text{ if } ab=0\\
a^\lambda b^{1-\lambda} &\text{ if } p=0\end{cases}$$
be the $\lambda$-weighted $p$-mean. The Borell-Brascamp-Lieb inequality asserts that if (again) $f,g,h\colon\mathbb{R}^d\to\mathbb{R}_{\geq 0}$ are integrable functions so that $\int_{\mathbb{R}^d}f(x)dx=\int_{\mathbb{R}^d}g(x)dx$, but which satisfy the weaker constraint that $h(\lambda x+(1-\lambda) y)\geq M_{-1/d,\lambda}(f(x),g(y))$ for all $x,y\in\mathbb{R}^d$, then still
$$\int_{\mathbb{R}^d}h(x)dx\geq \int_{\mathbb{R}^d}f(x)dx.$$

\subsection{Discrete Inequalities}
In \cite{van2025ruzsa}, the author with Keevash and Tiba established the following approximate Brunn-Minkowski inequality in the integers conjectured by Ruzsa \cite{ruzsa2006additive} using the technical framework from \cite{van2026locality}.
\begin{thm}[\cite{van2025ruzsa}]\label{ruzsaBM}
For all $d,\epsilon>0$, there exists a $n=n_{d,\epsilon}$ so that if $A,B\subset\mathbb{Z}^d$ are so that $B$ is not covered by $n$ parallel hyperplanes, then 
$$|A+B|^{1/d}\geq |A|^{1/d}+(1-\epsilon)|B|^{1/d}.$$
\end{thm}

The aim of this note will be to prove the following functional extension of this result in the spirit of the Borell-Brascamp-Lieb inequality. For a function $f\colon\mathbb{Z}^d\to\mathbb{R}_{\geq0}$ and $X\subset\mathbb{Z}^d$, we write $\sum_X f$ for $\sum_{x\in X} f(x)$ and $\sum f$ for $\sum_{x\in \mathbb{Z}^d} f(x)$. 

\begin{thm}\label{mainthm}
Let $p\in(0,1/d)$,  $f,g,h\colon\mathbb{Z}^d\to\mathbb{R}_{\geq0}$ so that $\sum f=\sum g$, for all $x,y\in\mathbb{Z}^d$, $h(x+y)\geq M_{-p,\frac12}(f(x),g(y))$, and for every $n=n_{d,\epsilon,p}$ parallel hyperplanes $H_1,\dots, H_n$, we have $\sum_{\bigcup_i H_i}f\leq \left(1-2^{d-\frac1p}\right)\sum f$.
Then
$$\sum h\geq (2^d-\epsilon)\sum f$$
\end{thm}
As $p\to 0$, this implies the following Pr\'ekopa-Leindler inequality

\begin{cor}
Let $f,g,h\colon\mathbb{Z}^d\to\mathbb{R}_{\geq0}$ so that $\sum f=\sum g$, for all $x,y\in\mathbb{Z}^d$, $h(x+y)\geq \sqrt{f(x)g(y)}$, and for every $n=n_{d,\epsilon,\alpha}$ parallel hyperplanes $H_1,\dots, H_n$, we have $\sum_{\bigcup_i H_i}f\leq \left(1-\alpha\right)\sum f$.
Then
$$\sum h\geq (2^d-\epsilon)\sum f.$$
\end{cor}

The non-degeneracy condition in \Cref{mainthm} is needed and asymptotically optimal in the following sense. Consider $g=\textbf{1}_{o}$ and $f=(1-\gamma)\textbf{1}_{o}+\frac{\gamma}{N^d}\textbf{1}_{[N]^d}$, where $[N]^d=\{0,\dots,N-1\}^d$ and $N$ arbitrarily large. Note that 
$$M_{-p,1/2}\left(1,\frac{\gamma}{N^d}\right)=\left(\frac{2}{1+\left(\frac{N^d}{\gamma}\right)^{p}}\right)^{1/p}<\left(\frac{2}{\left(\frac{N^d}{\gamma}\right)^{p}}\right)^{1/p} =2^{1/p}\frac{\gamma}{N^d}.$$
Hence, if we let $h=\textbf{1}_{o}+ 2^{1/p}\frac{\gamma}{N^d}\textbf{1}_{[N]^d}$, then $h(x+y)\geq M_{-p,\frac12}(f(x),g(y))$ and $\sum h= 1+ 2^{1/p}\gamma$ while $n$ hyperplanes contain at most a $1-\gamma+ O_{\gamma,d}(\frac{n}{N})$ proportion of $f$. To find the conclusion, we thus need $2^{1/p}\gamma+1\geq 2^d-\epsilon$.

 Some discrete version of the Pr\'ekopa-Leindler inequality have been proved previously in the style of the Ahlswede–Daykin inequality (or Four Function Theorem)
\cite{klartag2019poisson,gozlan2021transport,halikias2021discrete}. Others use rounding to get strong enough bounds on $h$ to prove Pr\'ekopa-Leindler type inequalities \cite{iglesias2020brunn,marsiglietti2024geometric}. Strong as these results are, the inequality presented here stays closer to the spirit of the Pr\'ekopa-Leindler inequality. In \cite{iglesias2020discrete}, Iglesias and Yepes establish a Borell-Brascamp-Lieb inequality in $\mathbb{Z}^n$ providing a lower bound based on removing the largest hyperplane sections of $f$.

\subsection{From Brunn-Minkowski to Borell-Brascamp-Lieb}
There are many proofs of the Borell-Brascamp-Lieb inequality and even more of the Pr\'ekopa-Leindler inequality. The approach presented here provides a general method to use Brunn-Minkowski inequalities to prove Borell-Brascamp-Lieb inequalities. The basic idea is to apply the Brunn-Minkowski inequality to pairs of level sets of the functions in question and to use an optimal transport map to find an appropriate pairing. This approach has been present in the field for some time in varying degrees of explicitness. Recently, Malliaris, Melbourne, Roberto, and Roysdon \cite{malliaris2025functional} explored the connection between these inequalities through an approach similar to \Cref{liftprop} below. In a short note, Cordero-Erausquin \cite{Cordero2025} clarified this connection further and provided some historical notes and simple proofs.

Consider domains $\Omega_1,\Omega_2,\Omega_3$ with a binary operation $S\colon \Omega_1\times \Omega_2\to\Omega_3$ and with a measures $\mu_1,\mu_2,\mu_3$ which allow a Brunn-Minkowski inequality, in the following sense. For non-empty measurable $A\subset \Omega_1$ and $B\subset\Omega_2$, write $S(A,B):=\{S(a,b)\in \Omega_3:a\in A,b\in B\}$ and assume $S(A,B)$ is also measurable, then we have:
\begin{equation}\label{pBM}
 \mu_3(S(A,B))^p\geq \lambda\mu_1(A)^p+(1-\lambda)\mu_2(B)^p   
\end{equation}
for some $p>0$ and $\lambda\in (0,1)$, or equivalently $\mu_3(S(A,B))\geq M_{p,\lambda}(\mu_1(A),\mu_2(B))$.

The traditional Brunn-Minkowski inequality is recovered for $\Omega_1=\Omega_2=\Omega_3=\mathbb{R}^d$, $S(x,y)=\lambda x+ (1-\lambda)y$, $p=1/d$, and $\mu_1=\mu_2=\mu_3$ is the Lebesgue measure. In most applications, we'll have $\Omega_1=\Omega_2=\Omega_3$ and $\mu_1=\mu_2=\mu_3$.

Operations and measures satisfying \eqref{pBM} allow a Borell-Brascamp-Lieb inequality of the following type. 
\begin{prop}\label{liftprop}
   Let $\Omega_1,\Omega_2,\Omega_3,$ $S$, $\mu_1,\mu_2,\mu_3$, $\lambda$, and $p$ satisfy \eqref{pBM}. Let $f\colon \Omega_1\to\mathbb{R}_{\geq0}$ be $\mu_1$-integrable, $g\colon \Omega_2\to\mathbb{R}_{\geq0}$ be $\mu_2$-integrable, and $h\colon \Omega_3\to\mathbb{R}_{\geq0}$ be $\mu_3$-integrable so that $\int_{\Omega_1} fd\mu_1=\int_{\Omega_2} g d\mu_2$ and so that for all $x\in \Omega_1,y\in\Omega_2$, we have $h(S(x,y))\geq M_{-p,\lambda}(f(x),g(y))$. Then $\int_{\Omega_3} h d\mu_3\geq \int_{\Omega_1} fd\mu_1$.
\end{prop}

\begin{proof}
Let $F_t:=\{x\in \Omega_1: f(x)>t\}$ and $G_t,$ $H_t$ analogously the superlevel sets of $f,g,$ and $h$, so that $\int_{0}^\infty \mu_1(F_t)dt=\int_{\Omega_1} f d\mu_1$. Note that by the lower bound on $h$, we have for any $a,b>0$ that $H_{M_{-p,\lambda}(a,b)}\supset S(F_a,G_b)$, so that 
$\mu_3(H_{M_{-p,\lambda}(a,b)})^p\geq \lambda\mu_1(F_a)^p+(1-\lambda)\mu_2(G_b)^p$. Let $T:\mathbb{R}_{\geq0}\to\mathbb{R}_{\geq0}$ be the transport map which pushes the distribution with density $t\mapsto \mu_1(F_t)$ to the distribution with density $t\mapsto \mu_2(G_t)$, so that $\int_0^s|F_t|dt=\int_0^{T(s)}|G_t|dt$ and $T'(t)=\frac{\mu_1(F_t)}{\mu_2(G_{T(t)})}$. We can then lower bound $\int_{\Omega_3} h d\mu_3$ as follows;
\begin{align*}
   \int_{\Omega_3} h d\mu_3&=\int_0^\infty \mu_3(H_t)dt =\int_0^\infty \mu_3\left(H_{M_{-p,\lambda}(t,T(t))}\right)dM_{-p,\lambda}(t,T(t))\\
   &\geq \int_0^\infty \left(\lambda\mu_1\left(F_t\right)^p+(1-\lambda)\mu_2\left(G_{T(t)}\right)^p\right)^{1/p}dM_{-p,\lambda}(t,T(t)).
\end{align*}
A simple computation (see e.g. \cite[Lemma 4.5]{figalli2025sharp}) shows that 
\begin{equation}\label{meanderiv}\frac{dM_{-p,\lambda}(t,T(t))}{dt}\geq \frac{1}{M_{p,\lambda}\left(1,\frac{1}{T'(t)}\right)}=\frac{1}{M_{p,\lambda}\left(1,\frac{\mu_2(G_{T(t)})}{\mu_1(F_t)}\right)}.
\end{equation}
Combining that with the fact that $\left(\lambda\mu_1\left(F_t\right)^p+(1-\lambda)\mu_2\left(G_{T(t)}\right)^p\right)^{1/p}=\mu_1(F_t)M_{p,\lambda}\left(1,\frac{\mu_2(G_{T(t)})}{\mu_1(F_t)}\right)$, we find that 
\begin{align*}
   \int_{\Omega_3} h d\mu_3&\geq \int_0^\infty \left(\lambda\mu_1\left(F_t\right)^p+(1-\lambda)\mu_2\left(G_{T(t)}\right)^p\right)^{1/p}dM_{-p,\lambda}(t,T(t))\geq \int_0^\infty \mu_1(F_t)dt=
   \int_{\Omega_1} fd\mu_1. \qedhere
\end{align*} 
\end{proof}

\begin{rmk}\label{pto0rem}Pr\'ekopa-Leindler type inequalities are obtained from the Borell-Brascamp-Lieb type inequalities by noting that if $h(S(x,y))\geq f(x)^\lambda g(y)^{1-\lambda}$, then in particular also $h(S(x,y))\geq M_{-p,\lambda}(f(x),g(y))$ for all $p>0$. In fact, if we have a bound of the type $\mu_3(S(A,B))\geq \mu_1(A)^\lambda\mu_2(B)^{1-\lambda}$, which can be seen as the limit of \eqref{pBM} as $p\to 0$, we still get a result akin to \Cref{liftprop} but with strengthened assumption $h(S(x,y))\geq \lim_{p\to 0}M_{-p,\lambda}(f(x),g(y))=f(x)^{\lambda}g(y)^{1-\lambda}$.
\end{rmk}

Though \eqref{pBM} is stated as an average, one can easily amend it for extra multiplicative constants. For instance, in $\mathbb{R}^d$ with $S(x,y)=x+y$, we have the classical Brunn-Minkowski inequality
$|A+B|\geq \left(|A|^{1/d}+|B|^{1/d}\right)^d=2^d \left(\frac12|A|^{1/d}+\frac12|B|^{1/d}\right)^d$. Hence, for $f,g,h\colon\mathbb{R}^d\to\mathbb{R}_{\geq0}$ with $\int f=\int g$ and $h(x+y)\geq \sqrt{f(x)g(y)}$, we find $\int h\geq 2^d\int f$. Of course, in this instance, one can easily see that replacing $h(x+y)\geq \sqrt{f(x)g(y)}$ with $h(\frac{x+y}{2})\geq \sqrt{f(x)g(y)}$ changes the integral by a factor $2^d$, but in general one might prefer either. For instance, in general groups $x+y$ is well-defined but $\frac{x+y}{2}$ might not exist. Meanwhile, in manifolds $\frac{x+y}{2}$ can be understood as a geodesic midpoint while $x+y$ might not have a natural interpretation.

\begin{rmk}\label{recoverrem}
    \Cref{liftprop} is sharp in the sense that one can recover the original underlying Brunn-Minkowski type inequality for $A$ and $B$ and the same power $p$. Indeed, let $f:=\frac{\textbf{1}_A}{\mu_1(A)}$ and $g:=\frac{\textbf{1}_B}{\mu_2(B)}$, so that $h=\textbf{1}_{S(A,B)}\cdot M_{-p,\lambda}\left(\frac{1}{\mu_1(A)},\frac{1}{\mu_2(B)}\right)=\frac{\textbf{1}_{S(A,B)}}{(\lambda\mu_1(A)^p+(1-\lambda)\mu_2(B)^p)^{1/p}} $. Since $\int_{\Omega_1} fd\mu_1=1=\int_{\Omega_2} g d\mu_2$, we conclude that $\frac{\mu_3(S(A,B))}{(\lambda\mu_1(A)^p+(1-\lambda)\mu_2(B)^p)^{1/p}}=\int_{\Omega_3}hd\mu_3\geq 1$, which was \eqref{pBM}.
\end{rmk}

\begin{rmk}
Those familiar with the Sobolev inequality, might recognize similarities with the derivation of the Sobolev inequality from the isoperimetric inequality. Indeed that derivation too allows for broad generalization to many of the examples considered below.
\end{rmk}

\subsection{Examples}

To illustrate the versatility of \Cref{liftprop} (and the underlying method), we consider a range of examples. These examples are provided for illustration and shouldn't be construed as assertions of novelty. The stated corollaries mostly follow directly from \Cref{liftprop}, but may require minimal variations on its proof.

\subsubsection{Gaussian measure}

The celebrated Gaussian Brunn-Minkowski inequality by Eskenazis and Moschidis \cite{eskenazis2021dimensional} asserts the following generalisation of the Brunn-Minkowski inequality to Gaussian space.
\begin{thm}[\cite{eskenazis2021dimensional}]\label{gaussianBM}
Let $\gamma_d$ be the Gaussian measure on $\mathbb{R}^d$, let $\lambda\in(0,1)$ and $K,L\subset\mathbb{R}^d$ convex and centrally symmetric, then
$$\gamma_n(\lambda K+ (1-\lambda)L)^{1/d}
\geq \lambda \gamma_n(K)^{1/d}+(1-\lambda)\gamma_n(L)^{1/d}.$$
\end{thm}

Recall that a function $f\colon\mathbb{R}^d\to\mathbb{R}_{\geq 0}$ is \emph{even} if $f(x)=f(-x)$ and \emph{quasiconcave} if all its level sets $F_t=\{x\in\mathbb{R}^d: f(x)>t\}$ are convex. In particular any log-concave function is quasiconcave. Using \Cref{liftprop}, \Cref{gaussianBM} has the following corollary (also obtained in e.g. \cite{malliaris2025functional,cordero2025concavity,aishwarya2025entropy}).

\begin{cor}\label{gaussiancor}
Let $\gamma_d$ be the Gaussian measure on $\mathbb{R}^d$, let $\lambda\in(0,1)$ and let $f,g,h:\mathbb{R}^d\to\mathbb{R}_{\geq0}$ be even, quasiconcave functions so that $\int_{\mathbb{R}^d}fd\gamma_d=\int_{\mathbb{R}^d}gd\gamma_d$ and for all $x,y\in\mathbb{R}^d$, $h(\lambda x+(1-\lambda)y)\geq M_{-1/d,\lambda}(f(x),g(y))$, then 
$$\int_{\mathbb{R}^d}hd\gamma_d\geq \int_{\mathbb{R}^d}fd\gamma_d$$
\end{cor}
This example highlights the additional strenght of the Borell-Brascamp-Lieb inequality over the Pr\'ekopa-Leindler inequality. The Pr\'ekopa-Leindler version of \Cref{gaussiancor} with the stronger condition $h(\lambda x+(1-\lambda)y)\geq f(x)^\lambda g(y)^{1-\lambda}$ is a direct consequence of the regular Pr\'ekopa-Leindler inequality (combined with the fact that the Gaussian density function is log-concave) without the need for \Cref{gaussianBM}.

\subsubsection{Dimensional Brunn-Minkowski conjecture}

The dimensional Brunn-Minkowski conjecture is a generalization of the Gaussian Brunn-Minkowski conjecture proposed by Gardner and Zvavitch \cite{gardner2010gaussian} suggesting that \Cref{gaussianBM} in fact holds for all even log-concave measures (see also \cite{livshyts2023universal,cordero2023improved} for partial results).

\begin{conj}\label{dimBMcon}
Let $\mu$ be an even log-concave measure on $\mathbb{R}^d$, let $\lambda\in(0,1)$ and $K,L\subset\mathbb{R}^d$ convex and centrally symmetric, then 
$$\mu(\lambda K+ (1-\lambda)L)^{1/d} \geq \lambda \mu(K)^{1/d}+(1-\lambda)\mu(L)^{1/d}.$$
\end{conj}
By \Cref{liftprop} this conjecture would imply and by \Cref{recoverrem} this conjecture is equivalent to the following conjecture.
\begin{conj}
Let $\mu$ be an even log-concave measure on $\mathbb{R}^d$, let $\lambda\in(0,1)$ and $f,g,h\colon\mathbb{R}^d\to\mathbb{R}_{\geq 0}$ quasiconcave and even, so that $\int_{\mathbb{R}^d}fd\mu=\int_{\mathbb{R}^d}gd\mu$ and $h(\lambda x+(1-\lambda)y)\geq M_{-1/d,\lambda}(f(x),g(y))$, then 
$$\int_{\mathbb{R}^d}hd\mu\geq\int_{\mathbb{R}^d}fd\mu.$$
\end{conj}

\subsubsection{Unconditional measures}
We say a set $A\subset\mathbb{R}^d$ is \emph{unconditional} if $a\in A$ implies that $\{x\in\mathbb{R}^d: |x_i|\leq |a_i|, \text{ for all } i=1,\dots,d\}\subset A$. A function is \emph{unconditional} if all of its superlevel sets are unconditional and a measure is \emph{unconditional} if its induced by an unconditional density function. Note that both the Euclidean and the Gaussian measure are unconditional. Because of their many symmetries unconditional sets have been a test case for many difficult problems. For instance, Saroglou \cite{saroglou2015remarks} settled the notorious log-Brunn-Minkowski conjecture for unconditional sets.
In \cite{livshyts2017brunn}, Livshyts, Marsiglietti, Nayar, and Zvavitch established the following Brunn-Minkowski type inequality for unconditional sets and unconditional measures. This result was later extended by Ritor\'e and Yepes Nicol\'as \cite{ritore2018}.
\begin{thm}[\cite{livshyts2017brunn}]\label{Livthm}
Let $\mu$ be an unconditional, product measure, $\lambda\in[0,1]$, and $A,B\subset \mathbb{R}^d$ measurable, unconditional, and non-empty, such that $\lambda A+(1-\lambda)B$ is also measurable, then
$$\mu(\lambda A+(1-\lambda)B)^{1/d}\geq \lambda\mu(A)^{1/d}+(1-\lambda)\mu(B)^{1/d}.$$
\end{thm}

By \Cref{liftprop}, we get a Borell-Brascamp-Lieb inequality for unconditional functions under unconditional measures.
\begin{cor}
Let $\mu$ be an unconditional product measure, $\lambda\in[0,1]$, and $f,g,h\colon\mathbb{R}^d\to\mathbb{R}_{\geq0}$ $\mu$-integrable and unconditional, such that $\int_{\mathbb{R}^d} fd\mu=\int_{\mathbb{R}^d} gd\mu$ and  $h\left(\lambda x +(1-
\lambda)y\right)\geq M_{-1/d,\lambda}(f(x),g(y))$ for all $x,y\in \mathbb{R}^d$, then
$$\int_{\mathbb{R}^d} hd\mu\geq\int_{\mathbb{R}^d} fd\mu.$$
\end{cor}

The crucial step in the previously mentioned resolution of the log-Brunn-Minkowski conjecture for unconditional sets by Saroglou \cite{saroglou2015remarks}, is a Brunn-Minkwoski type inequality for the following operation. Given $x,y\in \mathbb{R}_{\geq 0}^d$ and $\lambda\in (0,1)$, write $x^\lambda y^{1-\lambda}$ for the point $(x_1^\lambda y_d^{1-\lambda},\dots,x_d^\lambda y_d^{1-\lambda})$ and for unconditional sets $A,B\subset\mathbb{R}^d$, write $A^\lambda B^{1-\lambda}$ for the unconditional set defined by $\{x^\lambda y^{1-\lambda}: x\in A\cap\mathbb{R}^d_{\geq0},y\in B\cap\mathbb{R}^d_{\geq0}\}$ in the positive orthant. For unconditional convex sets $A,B$, we find that $A^{\lambda}B^{1-\lambda}$ is a subset of the geometric mean considered in the log-Brunn-Minkowski conjecture, so that the following suffices to settle this particular case of the log-Brunn-Minkowski inequality.

\begin{thm}[\cite{saroglou2015remarks}]\label{sarthm}
Let $A,B\subset\mathbb{R}^d$ unconditional and $\lambda\in (0,1)$, then
$$|A^\lambda B^{1-\lambda}|\geq |A|^\lambda |B|^{1-\lambda}.$$
\end{thm}

We immediately obtain the following corollary, using \Cref{liftprop} and \Cref{pto0rem}.

\begin{cor}\label{sarcor}
Let $\lambda\in (0,1)$ and $f,g,h\colon\mathbb{R}^d\to\mathbb{R}_{\geq 0}$ unconditional with $\int_{\mathbb{R}^d} fdx=\int_{\mathbb{R}^d} gdx$ and $h(x^\lambda y^{1-\lambda})\geq f(x)^{\lambda}g(y)^{1-\lambda}$ for all $x,y\in\mathbb{R}^d$, then
$$\int_{\mathbb{R}^d} fdx\geq\int_{\mathbb{R}^d} hdx.$$
\end{cor}

This corollary in fact also provides the induction step for an inductive proof of \Cref{sarthm} as follows.
\begin{proof}[Proof of \Cref{sarthm} and \Cref{sarcor}]
Proceed by induction on dimension $d$. For $d=1$, \Cref{sarthm} is trivial and \Cref{sarcor} follows from \Cref{liftprop} and \Cref{pto0rem}. For dimension $d$, consider the function $f:\mathbb{R}^{d-1}\to\mathbb{R}_{\geq0}, x\mapsto \mathcal{H}^1(A\cap (\{x\}\times\mathbb{R}))$ which measures the length of the fibre of $A$ above $x$, so that $\int_{\mathbb{R}^{d-1}}f(x)dx=|A|$. Analogously define $g,h$ for the sets $B$ and $A^\lambda B^{1-\lambda}$. Note that as all sets are unconditional, these functions satisfy $h(x^\lambda y^{1-\lambda})\geq f(x)^{\lambda}g(y)^{1-\lambda}$ for all $x,y\in\mathbb{R}^{d-1}$. Hence, by \Cref{sarcor} for dimension $d-1$, we find $|A^\lambda B^{1-\lambda}|=\int_{\mathbb{R}^{d-1}}h(x)dx\geq\int_{\mathbb{R}^{d-1}}f(x)dx=|A|$, so that \Cref{sarthm} holds for dimension $d$. \Cref{sarcor} for dimension $d$ follows by \Cref{liftprop} and \Cref{pto0rem} concluding the induction.
\end{proof}

\subsubsection{Kemperman's theorem}

Kemperman's theorem \cite{kemperman1964products} gives a very general, albeit from the perspective of the Brunn-Minkowski inequality weak, lower bound on the size of sumsets in a very general class of groups. Let $G$ be a unimodular, locally compact, connected\footnote{ \cite{satomi2022inequality} considers a relaxation of the connectedness condition, which in turn has its own functional corollary.} group with Haar measure $\mu$, then it says the following.
\begin{thm}
Let $A,B\subset G$ be measurable so that $\mu(A)+\mu(B)\leq \mu(G)$ and let $AB:=\{ab:a\in A, b\in B\}$ again measurable, then
$$\mu(AB)\geq \mu(A)+\mu(B).$$
\end{thm}

By \Cref{liftprop}, we obtain the following corollary.

\begin{cor}
Let $f,g,h\colon G\to\mathbb{R}_{\geq0}$ be integrable so that $\int_G f d\mu=\int_G gd\mu$, $\mu(supp(f))+\mu(supp(g))\leq \mu(G)$ and for all $a,b\in G$, $h(ab)\geq M_{-1,\frac12}(f(a),g(b))=\frac{2}{\frac{1}{f(a)}+\frac{1}{g(b)}}$, then
$$\int_G hd\mu\geq 2\int fd\mu.$$
\end{cor}

\subsubsection{Compact Lie groups}

Capturing the underlying dimension in lower bounds on the size of sumsets in wide classes of groups (and thus improving on Kemperman's theorem) has been the topic of several lines of research. In the context of compact Lie groups, Machado \cite{machado2024minimal} recently established a conjecture by Breuillard and Green about the minimal doubling of small sets in compact Lie groups. In fact, he proved the following Brunn-Minkowski type inequality for the size of the product set of small sets.

\begin{thm}[\cite{machado2024minimal}]\label{SimonThm}
Let $G$ be a compact connected Lie group and $\mu$ its Haar measure. Let $d:=d_G-d_H$ where $d_G$ is the dimension of $G$ and $d_H$ is the maximal dimension of a proper closed subgroup of $G$. For all $\epsilon>0$, there exists a $\delta>0$ so that if $A,B$ are compact subsets of $G$ with $\mu(A),\mu(B)\in (0,\delta]$, then
$$\mu(AB)^{1/d}\geq (1-\epsilon)\left(\mu(A)^{1/d}+\mu(B)^{1/d}\right).$$
\end{thm}
By \Cref{liftprop}, this theorem has the following Borell-Brascamp-Lieb type inequality for functions with small support in compact connected Lie groups as its corollary.

\begin{cor}Let $G,\mu,d$ as in \Cref{SimonThm}.
For all $\epsilon>0$, there exists a $\delta>0$ so that if $f,g,h\colon G\to\mathbb{R}_{\geq 0}$ have $\int_G f d\mu=\int_G gd\mu$, $h(xy)\geq M_{-d,1/2}(f(x),g(y))$ for all $x,y\in G$, and $\mu(supp(f)),\mu(supp(g))\leq \delta$, then
$$\int hd\mu\geq (1-\epsilon)2^d\int f d\mu.$$
\end{cor}

\subsubsection{Riemannian Manifolds}

This example differs from the others as the result that was originally proved was a Borell-Brascamp-Lieb type inequality \cite{cordero2001riemannian}, but it is included anyway for illustration. There has been extensive study of geometric inequalities on manifolds. For the sake of brevity, I will refer to the original paper by Cordero-Erausquin, McCann, and Schmuckenschl\"ager \cite{cordero2001riemannian} for precise definitions and conditions. Given a sufficiently well-behaved manifold $\mathcal{M}$ with geodesic distance $d$, define the following averaging operation. For $x,y\in \mathcal{M}$ let $S_\lambda(x,y):=\{m\in \mathcal{M}: d(x,m)=\lambda d(x,y) \text{ and }d(m,y)=(1-\lambda)d(x,y)\}$, the collection of midpoints between $x$ and $y$. For subsets $A,B\subset \mathcal{M}$, let $S_{\lambda}(A,B):=\bigcup_{x\in X, y\in Y}S_\lambda(x,y)$. One might expect the size of $S_{\lambda}(A,B)$ to depend on the curvature of $\mathcal{M}$ and the distance between the sets. The paper gives precise descriptions of this effect, but let me just give the following simple consequence.

\begin{thm}
Let $\mathcal{M}$ an $d$-dimensional manifold with non-negative curvature with measure $\mu$, and $\lambda\in(0,1)$. Let $A,B\subset \mathcal{M}$ measurable so that $S_{\lambda}(A,B)$ is also measurable, then
$$\mu(S_\lambda (A,B))^{1/d}\geq \lambda \mu(A)^{1/d}+(1-\lambda)\mu(B)^{1/d}.$$
\end{thm}

Once this theorem is established, it immediately gives the following Borell-Brascamp-Lieb type inequality. 

\begin{cor}
Let $\mathcal{M}$ an $d$-dimensional manifold with non-negative curvature with measure $\mu$ and $\lambda\in(0,1)$. Let $f,g,h:\mathcal{M}\to\mathbb{R}_{\geq0}$ so that $\int_\mathcal{M} fd\mu=\int_\mathcal{M} gd\mu$ and $\inf_{z\in S_\lambda(x,y)}h(z)\geq M_{-1/d,\lambda}(f(x),g(y))$. Then $$\int_\mathcal{M} hd\mu\geq \int_\mathcal{M} fd\mu.$$
\end{cor}
Note again that in \cite{cordero2001riemannian} a stronger version of this corollary is established, using more detailed information about the curvature. This example illustrates that to prove such strong results it often suffices to prove the associated Brunn-Minkowski type inequality.

\subsubsection{Discrete Cube}

Ollivier and Villani \cite{ollivier2012curved}  tried to extend the results from the setting of Riemannian manifolds to the discrete setting of the cube. Given three points $x,y,m_{x,y}\in \{0,1\}^d$, they say $m_{x,y}$ is a \emph{midpoint} of $xy$ if $d(x,y)=d(x,m_{x,y})+d(m_{x,y},y)$ and $|d(x,m_{x,y})-\frac{d(x,y)}{2}|\leq 1/2$, where $d(\cdot,\cdot)$ is the Hamming distance. Ollivier and Villani showed that the set of midpoints of satisfies the following Brunn-Minkowski type bound.

\begin{thm}[\cite{ollivier2012curved}]
Given $A,B\subset\{0,1\}^d$, let $M:=\{m\in\{0,1\}^d: m \text{ is the midpoint between } a\in A, b\in B\}$ be the collection of all midpoints, then
$|M|\geq \sqrt{|A||B|}$.
\end{thm}
By \Cref{liftprop}, we immediately obtain the following corollary.
\begin{cor}
Let $f,g,h\colon\{0,1\}^d\to\mathbb{R}_{\geq 0}$ so that $\sum f=\sum g$ and if $m$ is a midpoint of $x$ and $y$, then $h(m)\geq \sqrt{f(x)g(y)}$. Then
$$\sum h\geq \sum f.$$
\end{cor}
Ollivier and Villani in fact obtain stronger bounds based on the distance between the sets which, when examined carefully, will give even better bounds for the Pr\'ekopa-Leindler type consequence, but this is beyond the scope of this paper.

\section{Proof of the discrete Borell-Brascamp-Lieb inequality}
The proof will mimick the structure of the proof of \Cref{liftprop}, but some work is required to deal with the additional non-degeneracy condition in the Brunn-Minkowksi inequality (\Cref{ruzsaBM}) that fuels the proof.

We need the following two simple lemmata. First a trivial lower bound on the sup-convolution without non-degeneracy conditions.
\begin{lem}\label{trivialLB}
For any $f,g,h\colon\mathbb{Z}^{d-1}\to\mathbb{R}_{\geq 0}$, so that $h(x+y)\geq M_{-p,1/2}(f(x),g(y))$, we have $\sum h\geq M_{-p,1/2}\left(\sum f,\sum g\right)$
\end{lem}
\begin{proof}[Proof]
Consider $\mathfrak{m}:=\max\left\{ \max_{x\in\mathbb{Z}^{d-1}}\frac{f(x)}{\sum f},\max_{x\in\mathbb{Z}^{d-1}}\frac{g(x)}{\sum g}\right\}$. Assume without loss of generality that this maximum $\mathfrak{m}$ is attained by $g$ at the origin $o$. Then lower bound $h$ by noting that 
\begin{align*}
h(x)&\geq M_{-p,1/2}(f(x),g(o))=M_{-p,1/2}\left(\frac{f(x)}{\sum f}\sum f,\mathfrak{m}\sum g\right)\\
&\geq M_{-p,1/2}\left(\frac{f(x)}{\sum f}\sum f,\frac{f(x)}{\sum f}\sum g\right)=\frac{f(x)}{\sum f}M_{-p,1/2}\left(\sum f,\sum g\right).\end{align*}
The conclusion follows when summing $h(x)$ over all $x\in \mathbb{Z}^{d-1}$.
\end{proof}

Second a convexity lemma.

\begin{lem}\label{betaconvex}
Consider non-increasing sequence $\beta_i$, so that $\sum_{i=1}^\infty \beta_i=1$ and $\sum_{i=1}^n \beta_i\leq 1-\alpha$, then $\sum_{i=1}^\infty M_{-p,1/2}(c,\beta_i)\geq \alpha M_{-p,1/2}\left(\frac{cn}{1-\alpha},1\right).$
\end{lem}
\begin{proof}
    Note for all $i\geq n$, that $\beta_i\leq \frac{1-\alpha}{n}$ so $M_{-p,1/2}(c,\beta_i)=\beta_i M_{-p,1/2}(c\beta_i^{-1},1)\geq \beta_iM_{-p,1/2}\left(\frac{cn}{1-\alpha},1\right) $. Since $\sum_{i=n+1}^\infty \beta_i\geq \alpha $, the claim follows.
\end{proof} 

With these in place, we're ready to prove \Cref{mainthm}.

\begin{proof}[Proof of \Cref{mainthm}]
Let $\eta=\eta_{d,\alpha,\epsilon}>0$ be some parameter chosen small in terms of $d$, $\alpha$ and $\epsilon$.
By \Cref{ruzsaBM}, we can find a $m=m_{d,\eta}$ so that $|A+B|^{1/d}\geq (1-\eta)|A|^{1/d}+|B|^{1/d}$ if $A$ is not contained in $m$ hyperplanes.
Finally, take $n=n_{d,p,\epsilon,\alpha}$ large in terms of $d, p$, $\alpha$, $\epsilon$, $\eta$, and $m$.

Consider the level sets of $f, g$, and $h$,
$$F_{t}:=\{x\in\mathbb{Z}^d: f(x)>t\}, G_{t}:=\{x\in\mathbb{Z}^d: g(x)>t\}, \text{ and } H_{t}:=\{x\in\mathbb{Z}^d: h(x)>t\}.$$
Let $s_f$ and $s_g$ be the maxima of $f$ and $g$ respectively, so that
$$\int_0^{s_f}|F_t|dt=\sum f\text{ and }\int_0^{s_g}|G_t|dt=\sum g.$$
Let $T:[0,s_g]\to[0,s_f]$ be the transport map satisfying $\frac{d}{dt} T(t)=\frac{|G_t|}{|F_{T(t)}|}$, so that for all $t'\in [0,s_g]$, we have $\int_0^{t'}|G_t|dt=\int_0^{T(t')}|F_t|dt$.
Let 
$$t_{m,g}:=\sup\{t>0: G_t \text{ not contained in $m$ hyperplanes}\},$$
so that for all $t< t_{m,g}$, we have
$$|F_{T(t)}+G_{t}|^{1/d}\geq |F_{T(t)}|^{1/d}+(1-\eta)|G_t|^{1/d}\geq (1-\eta)\left(|F_{T(t)}|^{1/d}+|G_t|^{1/d}\right).$$
Note that by the definition of $h$, we have 
$H_{M_{-p,1/2}(T(t),t)}\supset F_{T(t)}+G_t$, so that
\begin{align*}
\int_{0}^{M_{-p,1/2}(T(t_{m,g}),t_{m,g})} |H_t|dt
&=\int_{0}^{t_{m,g}} \left|H_{M_{-p,1/2}(T(t),t)}\right|\frac{d\left(M_{-p,1/2}(T(t),t)\right)}{dt} dt\\
&\geq (1-\eta)^d \int_{0}^{t_{m,g}} \left(|F_{T(t)}|^{1/d}+|G_t|^{1/d}\right)^d\frac{d\left(M_{-p,1/2}(T(t),t)\right)}{dt} dt\\
&\geq (1-d\eta) \int_0^{t_{m,g}} 2^d|G_t|dt,
\end{align*}
where in the last inequality, we have used \Cref{meanderiv}.
Hence, if $\int_0^{t_{m,g}}|G_t|dt\geq (1-\eta) \int_0^{s_{g}}|G_t|dt=\sum f$, then we can conclude here.

We may thus assume henceforth that 
$\int_{t_{m,g}}^{s_g}|G_t|dt\geq \eta \sum g$, which in particular implies that
$\sum_{ G_{t_{m,g}}}g\geq \eta \sum g$. Since, by definition of $t_{m,g}$, the set $G_{t_{m,g}}$ is contained in $m$ hyperplanes, we can find a hyperplane $H\subset\mathbb{R}^d$ so that $\sum_{ H\cap\mathbb{Z}^d}g\geq \frac{\eta}{m} \sum g$. We shall show that even just the part of $g$ restricted to $H$ will have large doubling with $f$. 

Consider the set of hyperplanes parallel to $H$ intersecting $\mathbb{Z}^d$, and list them in the order $H_1,H_2,\dots$, so that $\sum_{\mathbb{Z}^d\cap H_i} f$ is non-increasing in $i$. Note that by the non-degeneracy condition of $f$, we have that 
$$\sum_{i=1}^{t_{n,f}} \sum_{\mathbb{Z}^d\cap H_i} f< (1-2^{d-\frac{1}{p}}) \sum f.$$

By \Cref{trivialLB}, we find that in the hyperplane $H+H_i$,
$$\sum_{\mathbb{Z}^d\cap(H+H_i)}h\geq M_{-p,1/2}\left(\sum_{\mathbb{Z}^d\cap H_i} f,\sum_{\mathbb{Z}^d\cap H} g\right)\geq M_{-p,1/2}\left(\frac{\sum_{\mathbb{Z}^d\cap H_i} f}{\sum f},\frac{\eta}{m}\right) \sum f.$$
Hence, 
$$\sum h \geq  \sum f\cdot \sum_{i=1}^{\infty} M_{-p,1/2}\left(\frac{\sum_{\mathbb{Z}^d\cap H_i} f}{\sum_{\mathbb{Z}^d} f},\frac{\eta}{m}\right).$$
By \Cref{betaconvex}, this implies 
$$\sum h \geq 2^{d-\frac{1}{p}}  M_{-p,1/2}\left(\frac{\eta n}{m\left(1-2^{d-\frac{1}{p}}\right)},1\right) \sum f.$$
Note that $M_{-p,1/2}\left(L,1\right)\to 2^{1/p}$ as $L\to \infty$, so for $n$ is sufficiently large in terms of $m,\eta,\epsilon$, and $p$, we find that indeed
$$\sum h\geq (2^d-\epsilon) \sum f.$$
\end{proof}

\section*{Acknowledgments}
The author is grateful to Gautam Aishwarya and Francisco Marin Sola for pointing me to \cite{malliaris2025functional} and \cite{Cordero2025}. The author would like to additionally thank Takashi Satomi for pointing out the missing connectedness condition in Kemperman's theorem and its corollary.

\bibliographystyle{alpha}
\bibliography{references}

\end{document}